\documentclass[10pt]{article}
\usepackage{amsmath, amsthm, amssymb}
\usepackage{fullpage}
\usepackage{enumerate}

\newtheorem{theorem}{Theorem}

\newtheorem{obs}[theorem]{Observation}

\begin{document}

\title{A Bipartite Graph That Is Not the $\gamma$-Graph of a Bipartite Graph}
\author{Christopher M. van Bommel\footnote{Supported by a Pacific Institute for the Mathematical Sciences Post-doctoral Fellowship.}\\Department of Mathematics\\University of Manitoba\\Winnipeg, MB, Canada\\\texttt{Christopher.VanBommel@umanitoba.ca}}

\maketitle

\begin{abstract}
For a graph $G = (V, E)$, the $\gamma$-graph of $G$ is the graph whose vertex set is the collection of minimum dominating sets, or $\gamma$-sets of $G$, and two $\gamma$-sets are adjacent if they differ by a single vertex and the two different vertices are adjacent in $G$.  An open question in $\gamma$-graphs is whether every bipartite graph is the $\gamma$-graph of some bipartite graph.  We answer this question in the negative by demonstrating that $K_{2, 3}$ is not the $\gamma$-graph of any bipartite graph.
\end{abstract}

\section{Introduction}

Reconfiguration is a topic in graph theory concerned with how two solutions to a problem on a particular graph are related.  We consider a sequence of steps that allow us to transform one solution to another, with each intermediate step also producing a solution.  We can represent this information in a reconfiguration graph, where the vertices represent the solutions and the edges represent that the solutions represented by their endpoints are separated by a single step.  Questions we can ask about reconguration graphs include structural properties (connectedness, Hamiltonicity, diameter, planarity), realisability (which graphs are the reconfiguration graphs of some graph), and algorithmic properties (how to move between solutions).  For a recent survey on the topic of reconfiguration, we refer the reader to~\cite{MN20}.

In this paper, we consider the reconfiguration of minimum cardinality dominating sets, or $\gamma$-sets.  Let $G$ be a graph and let $v$ be a vertex of $G$.  The \emph{open neighbourhood} of a vertex $v$, denoted $N(v)$, is the set of vertices adjacent to $v$, and the \emph{closed neighbourhood} of a vertex $v$, denoted $N[v]$, is $N(v) \cup \{v\}$.  For a subset of vertices $S$, we say $N(S) = \cup_{x \in S} N(x)$ and $N[S] = S \cup N(S)$.  A subset of vertices $D$ is a \emph{dominating set} of $G$ if $N[D] = V(G)$, that is, every vertex not in $D$ is adjacent to a vertex in $D$.  The \emph{domination number} of a graph $G$, denoted $\gamma(G)$, is the minimum cardinality of a dominating set of $G$.

For a dominating set $D$ of $G$ and a vertex $x \in D$, the set of \emph{private neighbours} of $x$, denoted $pn(x, D)$ is $N[x] - N[D - x]$, that is, the set of vertices in the closed neighbourhood of $x$ and not in the closed neighbourhood of any other vertex in $D$.  If $x \in pn(x, D)$, then $x$ is a \emph{$D$-self private neighbour} and if $y \neq x$ and $y \in pn(x, D)$, then $y$ is a \emph{$D$-external private neighbour} of $x$.  If $D$ is a $\gamma$-set, then every vertex in $D$ has a private neighbour.

Now, we form a reconfiguration graph, called a $\gamma$-graph, whose vertex set is the set of $\gamma$-sets.  Two possible steps to transform $\gamma$-sets have been considered.  In the first, which we denote $\mathcal{J}(G, \gamma)$, we say two $\gamma$-sets $D_1$ and $D_2$ are adjacent if and only if there exist vertices $x \in D_1$ and $y \in D_2$ such that $D_1 - \{x\} = D_2 - \{y\}$.  The $\gamma$-graph $\mathcal{J}(G, \gamma)$ is known as the \emph{$\gamma$-graph in the single vertex replacement adjacency model} or the \emph{jump $\gamma$-graph}.  In the second, which we denote $\mathcal{S}(G, \gamma)$, we say two $\gamma$-sets $D_1$ and $D_2$ are adjacent if and only if there exist {\bf adjacent} vertices $x \in D_1$ and $y \in D_2$ such that $D_1 - \{x\} = D_2 - \{y\}$.  The $\gamma$-graph $\mathcal{S}(G, \gamma)$ is known as the \emph{$\gamma$-graph in the slide adjacency model} or the \emph{slide $\gamma$-graph}.

Jump $\gamma$-graphs were introduced by Sridharan and Subramanaian~\cite{SS08} with the notation $\gamma \cdot G$.  Slide $\gamma$-graphs were introduced by Fricke, Hedetniemi, Hedetniemi, and Hutson~\cite{FHHH11} with the notation $G(\gamma)$.  We will mainly consider slide $\gamma$-graphs in this work.

Fricke et al.~\cite{FHHH11} demonstrated that every tree is the slide $\gamma$-graph of some graph, the slide $\gamma$-graphs of trees are connected and bipartite, and the slide $\gamma$-graphs of triangle-free graphs are triangle-free.  Moreover, they computed the slide $\gamma$-graphs for several classes of graphs, including complete graphs, complete bipartite graphs, paths, and cycles.  Connelly, Hedetniemi, and Hutson~\cite{CHH10} resolved the question of realisability of slide $\gamma$-graphs with the following.

\begin{theorem} \cite{CHH10}
Every graph is realisable as a $\gamma$-graph $\mathcal{S}(G, \gamma)$ of infinitely many graphs $G$.
\end{theorem}

Connelly et al.~\cite{CHH10} also considered the connectedness of slide $\gamma$-graphs; demonstrating all $\gamma$-graphs of graphs of order at most 5 are connected and characterizing the graphs of order 6 with disconnected $\gamma$-graphs.

Edwards, MacGillivray, and Nasserasr~\cite{EMN18} determined the following upper bounds on the order, diameter, and maximum degree of both jump and slide $\gamma$-graphs of trees, which answers open questions discussed by Fricke et al.~\cite{FHHH11}. A \emph{support vertex} is the unique neighbour of a vertex of degree one.

\begin{theorem} \cite{EMN18}
If $T$ is a tree of order $n$ having $s$ support vertices, then
\begin{enumerate}[(i)]
\item $\Delta(\mathcal{S}(T, y)) \le n - \gamma(T)$ and $\Delta(\mathcal{J}(T, \gamma)) \le n - \gamma(T)$,
\item $\mathrm{diam}(\mathcal{S}(T, \gamma)) \le 2 (2 \gamma(T) - s)$ and $\mathrm{diam}(\mathcal{J}(T, \gamma) \le 2 \gamma(T)$,
\item $|V(\mathcal{S}(T, \gamma))| = |V(\mathcal{J}(T, \gamma))| \le ((1 + \sqrt{13}) / 2)^{\gamma(T)}$.
\end{enumerate}
\end{theorem}

We see that the maximum degree and diameter of $\gamma$-graphs of trees are linear in the number of vertices.  Edwards et al.~\cite{EMN18} demonstrated that the bounds in (i) are sharp for an infinite family of trees, stated that no known tree has a $\gamma$-graph whose diamter exceeds half the bound in (ii), and showed that $|V(\mathcal{S}(T, \gamma)| > 2^{\gamma(T)}$ for infinitely many trees.  Lema\'{n}ska and \.{Z}yli\'{n}ski~\cite{LZ20} determined the following tight bounds on the diameter.  A support vertex is \emph{strong} if it is adjacent to at least two leaves, otherwise it is \emph{weak}.

\begin{theorem} \cite{LZ20}
If $T$ is a tree of order $n \ge 3$ having $s'$ weak support vertices and $s''$ strong support vertices, then $\mathrm{diam}(\mathcal{S}(T, \gamma)) \le \min\{2 (\gamma(T) - s'') - s', 2 (n - 1) / 3\}$ and $\mathrm{diam}(\mathcal{J}(T, \gamma)) \le \gamma(T) - s'$.
\end{theorem}

Mynhardt and Nasserasr~\cite{MN20} highlighted the following open questions:
\begin{enumerate}
\item \cite{FHHH11} Which graphs are $\gamma$-graphs of trees?
\item \cite{MT18} Is every bipartite graph the $\gamma$-graph of a {\bf bipartite} graph?
\end{enumerate}

An algorithm for finding the $\gamma$-graph of a tree and a simple characterization of {\bf trees} which are $\gamma$-graphs of trees were given by Finbow and van Bommel~\cite{FvB19}.  In what follows, we demonstrate that $K_{2, 3}$ is not the $\gamma$-graph of any bipartite graph, resolving the second question in the negative.

\section{Main Result}

We begin by presenting the following observations, noted by Finbow and van Bommel~\cite{FvB19} that will be used in the proof of the main result.

\begin{obs} \cite{FvB19} \label{spn}
If $D$ is a $\gamma$-set and if $u \in D$ is such that $pn(D, u) = \{u\}$, then for each $v \in N(u)$, we have that $D$ is adjacent to $D_v = (D - \{u\}) \cup \{v\}$.
\end{obs}

\begin{obs} \cite{FvB19} \label{bipn}
If $D_1$, $D_2 = (D_1 - \{u\}) \cup \{v\}$ and $D_3 = (D_1 - \{u\}) \cup \{w\}$ with $v, w \in N(u)$ are three distinct $\gamma$-sets of a bipartite graph\footnote{Finbow and van Bommel~\cite{FvB19} make this observation for trees, but a similar argument provides the result for bipartite graphs in general.}, then $pn(D_1, u) = pn(D_2, v) = pn(D_3, w) = \{u\}$.
\end{obs}

We can now demonstrate our main result.

\begin{theorem}
$K_{2, 3}$ is not the $\gamma$-graph of any bipartite graph.
\end{theorem}

\begin{proof}
Suppose $G$ is a bipartite graph and $G(\gamma) \cong K_{2, 3}$.  Let the $\gamma$-sets of $G$ be labelled $X_1, X_2, Y_1, Y_2, Y_3$ such that $\{X_1, X_2\} \cup \{Y_1, Y_2, Y_3\}$ is the bipartition of $K_{2,3}$.  Consider the $\gamma$-set $Y_1$.  Then $X_1 = (Y_1 \setminus \{a_1\}) \cup \{b_1\}$ and $X_2 = (Y_1 \setminus \{a_2\}) \cup \{b_2\}$, with $a_1 \sim b_1$ and $a_2 \sim b_2$.

Suppose $a_1 = a_2 = a$.  Then by Observation~\ref{bipn}, $pn(Y_1, a) = pn(X_1, b_1) = pn(X_2, b_2) = \{a\}$.  It follows by Observation~\ref{spn} that $\deg_G(a) = 2$.  We then need 
\[
Y_2 = (X_1 \setminus \{c_1\}) \cup \{d_1\} = (X_2 \setminus \{c_2\}) \cup \{d_2\} = (Y_1 \cup \{b_1, d_1\}) \setminus \{a, c_1\} = (Y_1 \cup \{b_2, d_2\}) \setminus \{a, c_2\}.
\]
We observe that $d_1 \neq a$, as otherwise $c_1 \in Y_1$ is adjacent to $a$, contradicting that $pn(Y_1, a) = \{a\}$.  By symmetry, $d_2 \neq a$, and by definition $a \neq b_1, b_2$.  Hence $a \notin Y_2$.  We also observe that $c_i \neq b_i$, as $b_i$ and $a$ have no common neighbours, and by definition $c_i \neq d_i$.  It therefore follows that $c_1 = c_2$, from which we see that $b_1 = d_2$ and $b_2 = d_1$.  It follows that there exists a set $Y$ such that $Y_1 = Y \cup \{a, c\}$, $Y_2 = Y \cup \{b, d\}$, $Y_3 = Y \cup \{e, f\}$, $X_1 = Y \cup \{b, c\}$, and $X_2 = Y \cup \{c, d\}$.  We then must have that $c \in \{e, f\}$, as otherwise $\{b, d\} = \{e, f\}$, which implies $Y_2 = Y_3$.  Without loss of generality, $f$ is a common neighbour of $b$ and $d$.  But then $Y_3$ contains no neighbour of $a$, contradicting that $Y_3$ is a $\gamma$-set.

Now suppose $b_1 = b_2 = b$.  We observe that $pn(Y_1, a_1) = \{a_1\}$ and $pn(Y_1, a_2) = \{a_2\}$.  But then $Y_1 \cup \{b\} \setminus \{a_1, a_2\}$ is a dominating set of smaller size, contradicting that $Y_1$ is a $\gamma$-set.

Finally, we may assume $a_1 \neq a_2$, $b_1 \neq b_2$.  Consider the $\gamma$-set $Y_2$.  Then $X_1 = (Y_2 \setminus \{c_1\}) \cup \{d_1\}$ and $X_2 = (Y_2 \setminus \{c_2\}) \cup \{d_2\}$.  By symmetry, we have that $c_1 \neq c_2$ and $d_1 \neq d_2$.  We observe that
\[
X_1 = (Y_1 \setminus \{a_1\}) \cup \{b_1\} = (Y_2 \setminus \{c_1\}) \cup \{d_1\}, \quad X_2 = (Y_1 \setminus \{a_2\}) \cup \{b_2\} = (Y_2 \setminus \{c_2\}) \cup \{d_2\},
\]
and hence we obtain that
\[
Y_2 = (Y_1 \setminus \{a_1, d_1\}) \cup \{b_1, c_1\} = (Y_1 \setminus \{a_2, d_2\}) \cup \{b_2, c_2\}.
\]
Suppose $a_1 = c_1$.  Then either $a_2 = c_2$ or $b_2 = d_2$.  Suppose $a_2 = c_2$, giving us
\[
Y_2 = (Y_1 \setminus \{d_1\}) \cup \{b_1\} = (Y_1 \setminus \{d_2\}) \cup \{b_2\}.
\]
It follows that $b_1 = b_2$, a contradiction.  Now suppose $b_2 = d_2$, giving us
\[
Y_2 = (Y_1 \setminus \{d_1\}) \cup \{b_1\} = (Y_1 \setminus \{a_2\}) \cup \{c_2\}.
\]
It follows that $d_1 = a_2$ and $b_1 = c_2$.  Hence, there exists a set $Y$ such that $Y_1 = Y \cup \{a_1, d_1\}$, $Y_2 = Y \cup \{a_1, b_1\}$, $X_1 = Y \cup \{b_1, d_1\}$, and $X_2 = Y \cup \{a_1, b_2\}$, so vertices $a_1, b_1, b_2, d_1$ induce a 4-cycle.  Now consider the sets $Z_1 = Y \cup \{b_1, b_2\}$ and $Z_2 = Y \cup \{b_2, d_2\}$.  It follows that $Z_1$ and $Z_2$ cannot both be $\gamma$-sets of $G$.  Assume without loss of generality that $Z_1$ is not a $\gamma$-set.  Then there exists a vertex $e$ such that $e \notin N[Z_1]$; in particular $e \nsim b_1$ and $e \nsim b_2$.  Now since $Y_2$ is a $\gamma$-set, we have $e \sim a_1$ and since $X_1$ is a $\gamma$-set, we have $e \sim d_1$.  But $a_1$ and $d_1$ are adjacent, so this contradicts that $G$ is bipartite.  Hence $a_1 \neq c_1$ and by symmetry, $a_2 \neq c_2$, $b_1 \neq d_1$, and $b_2 \neq d_2$.  

Thus, it follows that $c_1 = b_2$, $c_2 = b_1$, $d_1 = a_2$, and $d_2 = a_1$.  But then by symmetry, we must have that $Y_2$ and $Y_3$ are equal, a contradiction.  Hence $K_{2, 3}$ is not the $\gamma$-graph of any bipartite graph.
\end{proof}

We conclude by observing that the $\gamma$-graph of $C_4$ is $K_{2, 4}$, so $K_{2, 3}$ is not a forbidden (induced) subgraph for $\gamma$-graphs of bipartite graphs.  We leave open the question of characterizing which bipartite graphs are the $\gamma$-graph of a bipartite graph.

\bibliographystyle{plain}

\end{document}